\documentclass[10pt,a4paper,reqno]{amsart}

\usepackage{amsmath,amssymb,amsthm,amsfonts}
\usepackage{amsbsy}
\usepackage[utf8]{inputenc}
\usepackage[symbol]{footmisc}

\usepackage[english]{babel}
\usepackage{url}
\usepackage{wrapfig}
\usepackage{enumitem}
\usepackage{multicol}
\usepackage{moreenum}
\usepackage{subcaption}

\usepackage{tikz}

\usepackage[a4paper,left=25mm,right=25mm,top=30mm,bottom=30mm]{geometry}

\usepackage{xcolor}

\definecolor{LinkColor}{rgb}{0,0,0} %black
\usepackage[colorlinks=true,linkcolor=LinkColor,citecolor=LinkColor,urlcolor= LinkColor, naturalnames, hyperindex, pdfstartview=FitH, bookmarksnumbered, plainpages]{hyperref}

\newtheorem{theorem}{Theorem}[section]

\newtheorem{lemma}[theorem]{Lemma}
\newtheorem{proposition}[theorem]{Proposition}

\theoremstyle{definition}

\newtheorem{remark}[theorem]{Remark}

\newcommand{\Out}{\operatorname{Out}}

\newcommand{\Tr}{\operatorname{Tr}}

\newcommand{\V}{\textup{V}}
\newcommand{\C}{\textup{C}}
\newcommand{\N}{\textup{N}}
\newcommand{\ZZ}{\mathbb{Z}}
\newcommand{\QQ}{\mathbb{Q}}
\renewcommand{\O}{\operatorname{O}}

\usepackage[normalem]{ulem}

\title{An application of blocks to torsion units in group rings}

\author[A.~B\"achle]{Andreas B\"achle}
\address{Vakgroep Wiskunde, Vrije Universiteit Brussel, Pleinlaan 2, 1050 Brussels, Belgium.}
\email{\href{mailto:abachle@vub.be}{abachle@vub.be}, \href{mailto:leo.margolis@vub.be}{leo.margolis@vub.be}}

\author[L.~Margolis]{Leo Margolis}
\thanks{Both authors are postdoctoral researchers of the Research Foundation Flanders (FWO - Vlaanderen).}

\begin{document}

\maketitle

\begin{flushright}
\textit{One may dare to ask whether the theory of cyclic blocks can \\ provide additional
insight into Zassenhaus’ conjecture (ZC1). \\ We have no clue, but...}
\footnote[1]{from V. Bovdi, M. Hertweck, \textit{Zassenhaus Conjecture for central extensions of $S_5$}. Here (ZC1) denotes the First Zassenhaus Conjecture.}
\end{flushright}

\begin{abstract} We use the theory of blocks of cyclic defect to prove that under a certain condition on the principal $p$-block of a finite group $G$, the normalized unit group of the integral group ring of $G$ contains an element of order $pq$ if and only if so does $G$, for $q$ a prime different from $p$. 
Using this we verify the Prime Graph Question for all alternating and symmetric groups and also for two sporadic simple groups.  
\end{abstract}

\section{Introduction}

The influence of the structure of a finite group $G$ on the arithmetic properties of the unit group of the corresponding integral group ring $\mathbb{Z}G$ has given rise to many intriguing theorems and problems. To distinguish the arithmetic influenced by the group $G$ from that coming from the coefficient ring $\mathbb{Z}$, define $\V(\ZZ G)$ to be the group of units in $\mathbb{Z}G$ whose coefficients sum up to $1$. These are also called normalized units.
Such an influence is revealed, for example, if the orders of torsion elements in $\V(\mathbb{Z}G)$ coincide with the orders of elements in $G$. The \emph{Spectrum Problem}, which remains unanswered, asks whether this is indeed always the case. This fascinating problem is known to have a positive answer for solvable groups \cite{HertweckOrders}, but remains open in general. A stronger version of the Spectrum Problem, the First Zassenhaus Conjecture, was the focus of research for several decades, but recently first counterexamples were found \cite{EiseleMargolis17}. 

A graph encoding basic arithmetic properties of a (possibly infinite) group $A$ is the \emph{prime graph} $\Gamma(A)$: The vertices of $\Gamma(A)$ are labeled by the primes appearing as orders of elements in $A$ and two vertices $p$ and $q$ are connected by an edge if and only if $A$ contains an element of order $pq$. It follows from a fundamental theorem of G. Higman \cite[Theorem~12]{Higman1940Thesis} that the vertices of $\Gamma(G)$ and $\Gamma(\V(\ZZ G))$ coincide for each finite group $G$. Cohn and Livingstone proved that even the exponents of $G$ and $\V(\ZZ G)$ coincide \cite[Corollary 4.1]{CohnLivingstone}. The question if also the edges of the graphs coincide was put forward by Kimmerle \cite{Kimmerle2006}.  \\ 
 
\noindent\textbf{Prime Graph Question:}  For a finite group $G$, does $\Gamma(G) = \Gamma(\V(\ZZ G))$ hold?\\

In contrast to other problems concerning units in group rings there is a reduction result for the Prime Graph Question \cite[Theorem~2.1]{KimmerleKonovalov2016}:\\

\noindent
\textbf{Reduction:} The Prime Graph Question has a positive answer for $G$ if it has a positive answer for all almost simple images of $G$.\\
 
 Recall that a group $A$ is called \emph{almost simple} if it is sandwiched between a non-abelian simple group $S$ and its automorphism group, i.e., $S \simeq \operatorname{Inn}(S) \leq A \leq \operatorname{Aut}(S)$. In this case $S$ is called the socle of $A$.

Denote by $A_n$ and $S_n$ the alternating and symmetric group of degree $n$, respectively. In this note we prove the following.

\begin{theorem}\label{th:AltSym}
The Prime Graph Question has a positive answer for any almost simple group with socle isomorphic to some $A_n$. In particular, it has a positive answer for alternating and symmetric groups.
\end{theorem}

This will be a consequence of a more general result which states that the Prime Graph Question can be locally answered around the vertex $p$ assuming the principal $p$-block of $G$ lies in a particular Morita equivalence class. 

\begin{theorem}\label{prop:Brauer_Line} Let $G$ be a finite group and let $p$ be an odd prime. Assume that a Sylow $p$-subgroup of $G$ is of order $p$. Assume further there is exactly one conjugacy class of $p$-elements in $G$ and that the Brauer tree of the principal $p$-block of $G$ is a line.

If $q$ is a prime different from $p$, then $G$ contains an element of order $pq$ if and only if $\V(\ZZ G)$ does.
\end{theorem}

Yet Theorem~\ref{prop:Brauer_Line} suits perfectly for symmetric and alternating groups, we will see that it alone is also sufficient to obtain a result on the Prime Graph Question for two sporadic simple groups.

\begin{theorem}\label{th:Fi22HN}
The Prime Graph Question has a positive answer for any almost simple group that has the smallest of the sporadic Fischer groups $Fi_{22}$ or the Harada-Norton group $HN$ as socle. 
\end{theorem}

The proof of Theorem~\ref{prop:Brauer_Line} is based on a method developed by the authors \cite{BaechleMargolisEdinb, 4primaryII} and an application of this method \cite{MargolisConway} in conjunction with modular representation theory, especially the theory of blocks with cyclic defect. Theorem~\ref{prop:Brauer_Line} can be considered as a generalization of \cite[Theorem~D]{4primaryII}. 

Luthar and Passi were the first to consider the unit group of the integral group ring of a simple alternating group, in fact they proved the First Zassenhaus Conjecture for $A_5$ \cite{LutharPassi}. This paper gave rise to a method, extended by Hertweck \cite{HertweckBrauer}, and known nowadays as the HeLP-method. It was used by Salim \cite{Salim2007, Salim2011, Salim2013} to prove the Prime Graph Question for $A_n$ if $n \leq 10$, except for $n=6$. For the exception $A_6$, Hertweck showed the First Zassenhaus Conjecture \cite{HertweckA6} and a special argument from this paper, used to handle units of order $6$, inspired the above mentioned method by the authors. This method was first applied to prove the Prime Graph Question for all almost simple groups having $A_6$ as socle \cite{BaechleMargolisEdinb} (recall that the outer automorphism group of $A_6$ is an elementary-abelian group of order $4$). This completed the proof of the Prime Graph Question for groups with order divisible by exactly three pairwise different primes initiated in \cite{KimmerleKonovalov2015}.

Later Bächle and Caicedo \cite{BC} used the HeLP-method to prove the Prime Graph Question for symmetric and alternating groups of degree at most $17$. Moreover they achieved a generic result on the non-existence of units of certain orders. But attempts to use the HeLP-method to prove the Prime Graph Question for all alternating and symmetric groups have failed.

The Prime Graph Question is also known for almost-simple groups having socle $\operatorname{PSL}(2,p)$ or $\operatorname{PSL}(2,p^2)$  for any prime $p$ \cite{HertweckBrauer,4primaryI}, many almost-simple groups with order divisible by exactly four pairwise different primes \cite{4primaryII} and sixteen sporadic simple groups (thirteen were handled by Bovdi, Konovalov et.al., see for example \cite{BovdiKonovalovM24}, their automorphism groups and three more were treated in \cite{KimmerleKonovalov2015, MargolisConway}).   

The results and methods of this paper parallel some of the developments in the study of the so-called Second Zassenhaus Conjecture which asked if two group bases of an integral group ring are necessarily conjugate in the corresponding rational group algebra. Namely the symmetric groups were the first class of non-solvable groups for which this conjecture was known to have a positive answer \cite{Petersen76}. Moreover principal $p$-blocks with cyclic defect played a very prominent role in the later investigation of the conjecture for simple group, starting with \cite{BleherHissKimmerle}. This also provided the motivation for the epigraph.

\section{Preliminaries}

From here on $G$ will always denote a finite group and for an element $g \in G$ we denote the conjugacy class of $g$ in $G$ as $g^G$. For a subgroup $H$ of $G$, $\C_G(H)$ and $\N_G(H)$ denote its centralizer and normalizer in $G$, respectively. A fundamental notion in the study of torsion units in group rings is the so-called partial augmentation. Namely for an element $a = \sum_{g \in G} \alpha_g g \in \mathbb{Z}G$ we set
\[\varepsilon_g(a) =  \sum_{x \in g^G} \alpha_x. \]
This is called the \textit{partial augmentation of $u$ at $g$}. We collect some fundamental results on the partial augmentations of torsion units.

\begin{proposition}\label{prop:pa}
Let $u \in \V(\mathbb{Z}G)$ be a unit of finite order.
\begin{enumerate}[label=(\alph*)]
\item\label{BermanHigman} If $u \neq 1$ then $\varepsilon_1(u) = 0$ (Berman-Higman Theorem) \cite[Proposition 1.5.1]{JespersDelRioGRG}.
\item If $\varepsilon_g(u) \neq 0$ then the order of $g$ divides the order of $u$ \cite[Theorem 2.3]{HertweckBrauer}.
\end{enumerate}
\end{proposition}

Let $p$ be a prime and $a \in \mathbb{Z}G$. A ``folklore'' congruence states that

\[\varepsilon_x(a^p) \equiv \sum_{y^G,\ y^p \in x^G} \varepsilon_y(a) \mod p \]
where the sum runs over all those conjugacy classes of $G$ such that $y^p \in x^G$, cf.\ e.g.\ \cite[Remark 6]{BovdiHertweck}. Together with Proposition~\ref{prop:pa}.\ref{BermanHigman} this yields the following result which will be very useful to us.

%From a further "folklore" congruence, cf. e.g. \cite[Remark 6]{BovdiHertweck}, and the preceding proposition we get a congruence which will be very useful to us.

\begin{lemma}\label{lemma:wagner}
Let $p$ be a prime, $u \in \V(\mathbb{Z}G)$ a unit of finite order and let $g_1$,...,$g_k$ be representatives of the conjugacy classes of elements of order $p$ in $G$. Then
\[\sum_{i=1}^k \varepsilon_{g_i}(u) \equiv 0 \mod p. \] 
\end{lemma}

If $u$ is a unit of finite order $n$ in $\V(\mathbb{Z}G)$ and $D$ an ordinary representation of $G$, then $D$ extends linearly to $\mathbb{Z}G$ and $D(u)$ is an invertible matrix of order dividing $n$. Hence $D(u)$ is a diagonalizable matrix whose eigenvalues are $n$-th complex roots of unity. For such an $n$-th root of unity $\xi$ we will denote the multiplicity of $\xi$ as an eigenvalue of $D(u)$ as $\mu(\xi, u, \chi)$, where $\chi$ is the character of $D$.

At the heart of the proof of our Theorem~\ref{prop:Brauer_Line} lies the following reformulation of \cite[Theorem~1.2]{MargolisConway}. Its proof uses the method introduced by the authors in \cite{BaechleMargolisEdinb,4primaryII}. For an ordinary character $\chi$ denote by $\mathbb{Q}(\chi)$ the smallest field containing all the character values of $\chi$.
%, analyzing the combinatorics involved for the described decomposition behavior.

\begin{theorem}\label{prop:ineqmultiplicities} Let $G$ be a finite group and $p$ be an odd prime. Assume that some $p$-block of $G$ is a Brauer Tree Algebra whose Brauer tree is of the form
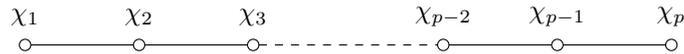
\[
\begin{tikzpicture}
\node[label=north:{$\chi_1$}] at (0,1.5) (1){};
\node[label=north:{$\chi_{2}$}] at (1.5,1.5) (2){};
\node[label=north:{$\chi_{3}$}] at (3,1.5) (3){};
\node[label=north:{$\chi_{p-2}$}] at (5.5,1.5) (4){};
\node[label=north:{$\chi_{p-1}$}] at (7,1.5) (5){};
\node[label=north:{$\chi_{p}$}] at (8.5,1.5) (6){};
\foreach \p in {1,2,3,4, 5, 6}{
\draw[fill=white] (\p) circle (.075cm);
}
\draw (.075,1.5)--(1.425,1.5);
\draw (1.575,1.5)--(2.925,1.5);
\draw[dashed] (3.075,1.5)--(5.425,1.5);
\draw (5.575,1.5)--(6.925,1.5);
\draw (7.075,1.5)--(8.425,1.5);
\end{tikzpicture}
\] 
for ordinary characters $\chi_1$,...,$\chi_p$. 
Assume that in the rings of integers of $\mathbb{Q}(\chi_1)$,...,$\mathbb{Q}(\chi_p)$ the prime $p$ is unramified.

Let $u \in \V(\ZZ G)$ be a unit of order $pm$, for a positive integer $m$ not divisible by $p$. Let $\xi$ be some $m$-th root of unity and let $\zeta_p$ be a primitive complex  $p$-th root of unity.
Then 
\[\sum_{i=1}^{p} (-1)^i \mu(\xi \cdot \zeta_p, u, \chi_i)\ \leq\ \mu(\xi, u, \chi_1). \] 
\end{theorem}

In fact the multiplicities of eigenvalues of a torsion unit $u$ under a representation can be computed from the values of the corresponding character and the partial augmentations of $u$ via the orthogonality relations, cf.\ e.g.\ \cite{LutharPassi}. We record this for later use. Here, $\operatorname{Tr}_{K/\mathbb{Q}} \colon K \to \mathbb{Q}$ denotes the $\mathbb{Q}$-linear number theoretic trace of the Galois extension $K/\mathbb{Q}$, i.e., the sum of all the Galois conjugates of the element.

\begin{lemma}\label{prop:HeLP}
Let $u \in \V(\mathbb{Z}G)$ be a unit of order $n$ and let $D$ be an ordinary representation of $G$ with character $\chi$. For a positive integer $k$ denote by $\zeta_k$ a primitive complex $k$-th root of unity. Then the multiplicity of an $n$-th root of unity $\xi$ as an eigenvalue of $D(u)$ is given by
\begin{align}\label{HeLP-restrictions}
\mu(\xi, u, \chi) = \frac{1}{n} \sum_{\substack{d \mid n}} \operatorname{Tr}_{\mathbb{Q}(\zeta_{n/d})/\mathbb{Q}}(\chi(u^d)\xi^{-d}).
\end{align}
\end{lemma}

Finally we will also need a theorem of Brauer \cite[Corollary~5]{BrauerAppl} which was used by Glauberman in the proof of his famous $Z^*$-theorem. We give it in a slightly more general form which can be found in \cite[(7.7) Theorem]{NavarroBook}. For a prime $p$ we will use the standard notation $\O_{p'}(G)$ to denote the largest normal subgroup of $G$ whose order is not divisible by $p$.

\begin{theorem}\label{th:Glauberman}
Let $p$ be a prime, $g$ a $p$-element in $G$ and $h$ an element in $\O_{p'}(\C_G(g))$. Let $\chi$ be a character in the principal $p$-block of $G$. Then $\chi(g) = \chi(gh)$.
\end{theorem}

\section{Principal Blocks with Defect $1$}

We will now embark on the proof of Theorem~\ref{prop:Brauer_Line}.% which gives a positive answer to the Prime Graph Question ``at the vertex $p$'' in case the Sylow $p$-subgroup of $G$ has order $p$ and the Brauer tree of the principal $p$-block behaves nicely. 

\begin{proof}[Proof of Theorem~\ref{prop:Brauer_Line}] Firstly, if $(p, q) = (3, 2)$, then there is an element of order $6$ in $\V(\ZZ G)$ if and only if there is an element of order $6$ in $G$, by \cite[Theorem~D]{4primaryII}. So we may always assume that $(p, q) \not= (3, 2)$.

Assume there is no element of order $pq$ in $G$ but there is a unit of order $pq$ in $\V(\ZZ G)$. Let $y$ be an element of order $p$ in $G$. As there is exactly one conjugacy class of elements of order $p$ in $G$ it follows that $[\N_G(\langle y\rangle ):\C_G(y)] = p-1$ and that there are exactly $p$ irreducible complex characters in the principal $p$-block $B_0$ \cite[(11.1) Theorem]{NavarroBook}, say $\chi_1,...,\chi_p$, where $\chi_1 = \textbf{1}$ is the principal character. We will use the theory of blocks of cyclic defect and Brauer trees as described e.g.\ in \cite{FeitRep}. W.l.o.g. we can assume that the Brauer tree of $B_0$ is of the form
\begin{equation}\label{brauer_tree_principal_block}
\begin{tikzpicture}
\node[label=north:{$\chi_1 = \textbf{1}$}] at (0,1.5) (1){};
\node[label=north:{$\chi_{2}$}] at (1.5,1.5) (2){};
\node[label=north:{$\chi_{3}$}] at (3,1.5) (3){};
\node[label=north:{$\chi_{p-2}$}] at (5.5,1.5) (4){};
\node[label=north:{$\chi_{p-1}$}] at (7,1.5) (5){};
\node[label=north:{$\chi_{p}$}] at (8.5,1.5) (6){};
\foreach \p in {1,2,3,4, 5, 6}{
\draw[fill=white] (\p) circle (.075cm);
}
\draw (.075,1.5)--(1.425,1.5);
\draw (1.575,1.5)--(2.925,1.5);
\draw[dashed] (3.075,1.5)--(5.425,1.5);
\draw (5.575,1.5)--(6.925,1.5);
\draw (7.075,1.5)--(8.425,1.5);
\end{tikzpicture}.
\end{equation}

Note that by our assumptions the character values of all irreducible complex characters of $G$ are $p$-rational, i.e. $p$ is unramified in every ring of character values of irreducible ordinary characters in the principal block of $G$.

Define $\nu$ to be the alternating sum of the ordinary irreducible characters in $B_0$, \[\nu(a) = \sum_{i = 1}^p (-1)^{i} \chi_i(a), \qquad a \in \ZZ G. \]
 
We will show that in our situation $\nu(g) = 0$ for all $ g \in G$ and obtain a contradiction with the linear independence of the irreducible ordinary characters of $G$. Recall that a torsion unit in $\V(\ZZ G)$ is called $p$-regular if it has order not divisible by $p$ and $p$-singular otherwise. In particular, an element $g \in G$ is $p$-regular precisely if it is a $p'$-element. We now proceed to show that $\nu$ evaluates to zero in several cases.

 \begin{enumerate}[label=(\alph*)]
 \item \label{p-regular} \emph{For $p$-regular units of $\V(\mathbb{Z}G)$.}
 Let $\psi_1, ..., \psi_{p-1}$ be the irreducible Brauer characters in $B_0$ corresponding to the edges of the Brauer tree in \eqref{brauer_tree_principal_block} labeling the edges from left to right. Let $x$ be a $p'$-element of $G$. Then $\chi_1(x) = \psi_1(x)$,\ $\chi_p(x) = \psi_{p-1}(x)$,\ and $\chi_i(x) = \psi_{i-1}(x) + \psi_i(x)$ for every $i \in \{2, ..., p-1\}$. So we get 
 \[ \nu(x)\ =\ \sum_{i = 1}^p (-1)^{i} \chi_i(x)\ =\ -\psi_1(x) + \sum_{i = 2}^{p-1} \left( (-1)^{i} \psi_{i-1}(x) + \psi_i(x) \right) -  \psi_{p-1}(x)\ =\ 0, \]
 cf.\ \cite[Chapter VII, Theorem 2.15 (iii)]{FeitRep}.

Assume now that $v \in \V(\ZZ G)$ is $p$-regular. Then $\varepsilon_h(v) = 0$ for all $p$-singular elements $h \in G$ by Proposition~\ref{prop:pa} and 
\[\chi(v) = \sum_{x^G,\ p \nmid o(x)} \varepsilon_x(v) \chi(x)\]
for any ordinary character $\chi$, where the sum runs over all $p$-regular conjugacy classes of $G$. Thus 
\[ \nu(v) \ = \ \sum_{i = 1}^p (-1)^{i} \chi_i(v)\ =\ \sum_{i = 1}^p (-1)^{i} \sum_{x^G,\ p \nmid o(x)} \varepsilon_x(v) \chi_i(x)\ =\ \sum_{x^G,\ p \nmid o(x)} \varepsilon_x(v) \sum_{i = 1}^p (-1)^{i} \chi_i(x)\ =\ 0. \]

 \item \emph{For $p$-elements of $G$.}\label{p-elements} Assume that $u \in \V(\ZZ G)$ is an element of order $pq$. For a positive integer $k$ denote by $\zeta_k$ a primitive complex $k$-th root of unity. Note that since $\chi_i$ is $p$-rational for every $i \in \{1,...,p\}$ we have $\Tr_{\QQ(\zeta_p)/\QQ}(\chi_i(u^q) \zeta_p) = -\chi_i(u^q)$.	
 Thus by Theorem~\ref{prop:ineqmultiplicities} with $\xi = 1$, Lemma~\ref{prop:HeLP} using \ref{p-regular} we have
 \begin{align*} 1 &= \mu(1, u, \chi_1) \\
  &\geq\ \sum_{i=1}^p (-1)^i \mu(\zeta_p, u, \chi_i) \\
 &= \frac{1}{pq} \sum_{i = 1}^p (-1)^i \Big( \chi_i(u^{pq}) + \Tr_{\QQ(\zeta_q) / \QQ}(\chi_i(u^p)) + \Tr_{\QQ(\zeta_p) / \QQ}(\chi_i(u^q) \zeta_p^{-q}) \\
 & \phantom{= \frac{1}{pq} \sum_{i = 1}^p (-1)^i \Big(}  +  \Tr_{\QQ(\zeta_{pq}) / \QQ}(\chi_i(u) \zeta_p^{-1})\Big) \\
\end{align*}
\begin{align} \begin{split}
  & \ =\ \frac{1}{pq}\Bigg( \sum_{i=1}^{p} (-1)^i \chi_i(1) + \Tr_{\QQ(\zeta_q)/\QQ}\left( \sum_{i=1}^{p} (-1)^i \chi_i(u^p)\right) - \sum_{i=1}^{p} (-1)^i \chi_i(u^q) 
 \\ & \phantom{\ =\ \frac{1}{pq}\Bigg(}  + \Tr_{\QQ(\zeta_{pq})/\QQ}\left(\sum_{i=1}^{p} (-1)^i  \chi_i(u)\zeta_{p}^{-1}\right)  \Bigg) 
 \\ & \ =\ \frac{1}{pq} \left( \nu(1) + \Tr_{\QQ(\zeta_q)/\QQ}\left( \nu(u^p) \right) - \nu(u^q) + \Tr_{\QQ(\zeta_{pq})/\QQ}\left(\sum_{i=1}^{p} (-1)^i  \chi_i(u)\zeta_{p}^{-1}\right) \right)
 \\ & \ =\ \frac{1}{pq} \left( - \nu(u^q)  + \Tr_{\QQ(\zeta_{pq})/\QQ}\left(\sum_{i=1}^{p} (-1)^i  \chi_i(u)\zeta_{p}^{-1}\right)  \right). \end{split} \label{eq:ineq_xi1_1} 
\end{align}
 
 For $i \in \{1, ..., p\}$ write \[\chi_i(u) = \varepsilon_y(u)\chi_i(y) + \sum_{j=1}^s \varepsilon_{x_j}(u) \chi_i(x_j),\] where $y \in G$, as above, is an element of order $p$ and $x_1, ..., x_s$ is a set of representatives of the conjugacy classes of $G$ consisting of elements of order $q$. This is possible by Proposition~\ref{prop:pa} and the assumption that $G$ does not contain elements of order $pq$. Since $G$ only has one conjugacy class of elements of order $p$, $\chi_i(u^q) = \chi_i(y)$ and this character value is an integer for all $i \in \{1, ..., p\}$. Hence from \eqref{eq:ineq_xi1_1},
 \begin{align*} 
 pq \ & \geq\ - \nu(y)  + \sum_{i=1}^{p} (-1)^i \Tr_{\QQ(\zeta_{pq})/\QQ}\left(\varepsilon_y(u)\chi_i(y) \zeta_{p}^{-1}\right) \\
  & \phantom{\geq\ -} + \Tr_{\QQ(\zeta_{pq})/\QQ}\left(\sum_{j=1}^s \varepsilon_{x_j}(u)\zeta_p^{-1} \sum_{i=1}^{p} (-1)^i \chi_i(x_j) \right) \\
 & =\ -\nu(y) + \nu(y) \varepsilon_y(u) \Tr_{\QQ(\zeta_{pq})/\QQ}\left(\zeta_{p}^{-1}\right) + \Tr_{\QQ(\zeta_{pq})/\QQ}\left(\sum_{j=1}^s \varepsilon_{x_j}(u)\zeta_p^{-1}\nu(x_j) \right).
 \end{align*} 
 As $x_1, ..., x_s$ are $p$-regular the last summand evaluates to zero by \ref{p-regular}. Since $\Tr_{\QQ(\zeta_{pq})/\QQ}\left(\zeta_{p}^{-1}\right) = -(q-1)$ we get
 \[ pq\ \geq\ \Big(-1 - (q-1)\varepsilon_y(u) \Big) \nu(y). \]
 
 On the other hand, assume w.l.o.g. $\zeta_p\zeta_q = \zeta_{pq}$. By Theorem~\ref{prop:ineqmultiplicities} with $\xi = \zeta_q$, Lemma~\ref{prop:HeLP} and since $1$ and $u^p$ are $p$-regular we get similar as in \eqref{eq:ineq_xi1_1}: 
\begin{align*} 
  0\ =\ & \mu(\zeta_q, u, \chi_1) \\
   \geq\ & \sum_{i=1}^p (-1)^i \mu(\zeta_q\zeta_p,u,\chi_i) \\
   =\ & \frac{1}{pq} \sum_{i=1}^{p} (-1)^i \Big(\chi_i(1)  + \Tr_{\QQ(\zeta_q)/\QQ}\left(\chi_i(u^p)\zeta_{q}^{-p}\right) + \Tr_{\QQ(\zeta_p)/\QQ}(\chi_i(u^q)\zeta_p^{-q}) \\
  & \phantom{\frac{1}{pq} \sum_{i=1}^{p} (-1)^i \Big(} + \Tr_{\QQ(\zeta_{pq})/\QQ}\left(\chi_i(u)\zeta_{pq}^{-1}\right) \Big) \\ 
 =\ & \frac{1}{pq}\left(\nu(1) + \Tr_{\QQ(\zeta_q)/\QQ}(\nu(u^p)\zeta_q^{-p}) - \nu(u^q) + \Tr_{\QQ(\zeta_{pq})/\QQ}\left(\sum_{i=1}^p (-1)^i\chi_i(u)\zeta_{pq}^{-1}\right)  \right)  \\
 =\ & \ \frac{1}{pq} \left( - \nu(u^q)  + \Tr_{\QQ(\zeta_{pq})/\QQ}\left(\sum_{i=1}^{p} (-1)^i \chi_i(u)\zeta_{pq}^{-1}\right) \right). 
 \end{align*}
 Again expressing $\chi_i(u) = \varepsilon_y(u)\chi_i(y) + \sum_{j=1}^s \varepsilon_{x_j}(u) \chi_i(x_j)$ with $y$ of order $p$ and $x_1, ..., x_s$ all $p$-regular we obtain
\begin{align*}
  0 \ & \geq\  \Big( - \nu(y)  + \nu(y)\Tr_{\QQ(\zeta_{pq})/\QQ}\left(\varepsilon_y(u) \zeta_{pq}^{-1}\right) \Big) + \Tr_{\QQ(\zeta_{pq})/\QQ}\left(\sum_{j=1}^s \varepsilon_{x_j}(u)\zeta_{pq}^{-1} \nu(x_j) \right) \\
  & =\ (\varepsilon_y(u) - 1) \nu(y). 
   \end{align*}
 Summarizing we have \begin{align} \Big(\varepsilon_y(u) - 1\Big) \nu(y) \ \leq\ 0 \qquad \text{and} \qquad  \Big(1 + (q-1)\varepsilon_y(u) \Big) \nu(y) \ \geq \ -pq. \label{eq:inequalities} \end{align}
  
 By Lemma~\ref{lemma:wagner}, $\varepsilon_y(u) \equiv 0 \bmod p$ and $\varepsilon_y(u) = 1 - \sum_{x^G,\ o(x) = q} \varepsilon_x(u) \equiv 1 \bmod q$. Hence either \[ \varepsilon_y(u) \geq p  \qquad \text{or} \qquad \varepsilon_y(u) \leq -p. \]
  
 Assume first $\varepsilon_y(u) \geq p$. Then by the first inequality in \eqref{eq:inequalities} we get $\nu(y) \leq 0$. The second inequality in \eqref{eq:inequalities} yields 
 \[ \nu(y)\ \geq\ - \frac{pq}{1 + (q-1)p}\ =\ -1 - \frac{p-1}{1 + (q-1)p}\ >\ -2  \]

 Assume now $\varepsilon_y(u) \leq -p$. Then  by the first inequality in \eqref{eq:inequalities} we get $\nu(y) \geq 0$ and the second inequality in \eqref{eq:inequalities} gives 
 \[ \nu(y)\ \leq\ \frac{pq}{-1 + (q-1)p} < p, \] since $(p, q) \not= (3,2)$. 
 
 Taken together we obtain
 \begin{align}\label{eq:CompleteIneq}
 -p \ < \ \nu(y)\ <\ p.
 \end{align}
 
Let $D$ be an ordinary representation of $G$ with character $\chi$. As $y$ is rational in $G$ and of order $p$ each primitive $p$-th root of unity appears with the same multiplicity in the eigenvalues of $D(y)$. If in the eigenvalues of $D(y)$ each primitive $p$-th root of unity is replaced by $1$ then $\chi(y)$ increases by $p \cdot \mu(\zeta_p,y,\chi)$.
Hence $\chi(y) \equiv \chi(1) \bmod p$ for any ordinary character $\chi$. So $0 = \nu(1) \equiv \nu(y) \bmod p$ and in any case by \eqref{eq:CompleteIneq} we have $\nu(y)=  0$, as desired. 

 \item \emph{For arbitrary $p$-singular elements of $G$.} Let $h \in G$ be a $p$-singular element in $G$. Then $h$ has order $pm$ for some positive integer $m$ not divisible by $p$. Write $h = h_ph_{p'}$, where $h_p$ and $h_{p'}$ denote the $p$-part and the $p'$-part of $h$, respectively. Then $h_p$ is conjugate in $G$ to $y$ and $\C_G(h_p) = \langle h_p \rangle \times Q$, for a $p'$-subgroup $Q$ of $G$, by Burnside's $p$-complement theorem. Hence the $p'$-part $h_{p'}$ of $h$ is contained in $Q = \O_{p'}(\C_G(h_p))$. By Theorem~\ref{th:Glauberman} and part \ref{p-elements} we obtain
 \[ \nu(h)\ =\ \sum_{i = 1}^p (-1)^{i} \chi_i(h)\ =\ \sum_{i = 1}^p (-1)^{i} \chi_i(h_ph_{p'})\ =\ \sum_{i = 1}^p (-1)^{i} \chi_i(y)\ =\ \nu(y)\ =\ 0. \]
 
\end{enumerate}
We conclude $\nu(g) = 0$ for all $g \in G$, the promised contradiction.
\end{proof}

\section{Applications}

We can now apply Theorem~\ref{prop:Brauer_Line} to groups where the principal $p$-block for certain primes is well behaved. This is in particular the case for alternating and symmetric groups of degree $n$ and primes $p > \frac{n}{2}$ and suffices to obtain a positive answer for the Prime Graph Question in these cases.

\begin{proof}[Proof of Theorem~\ref{th:AltSym}] The outer automorphism group of a simple alternating group of degree $n$ is isomorphic to a cyclic group of order $2$, except if $n=6$. As the Prime Graph Question has a positive answer in case $n \leq 17$ by \cite[Theorem~5.1]{BC}, we hence may assume that $n \neq 6$ and $\Out(A_n) \simeq C_2$. Let $q < p \leq n$ be primes.

First consider symmetric groups and assume that $S_n$ does not contain an element of order $pq$. Then $p > \frac{n}{2}$ and a Sylow $p$-subgroup of $S_n$ is of order $p$. Since symmetric groups are rational, all ordinary and Brauer characters are rational-valued. Thus by \cite[VII, Theorem~9.2]{FeitRep} the Brauer tree of the principal $p$-block of $S_n$ is a line and hence $\V(\ZZ S_n)$ also does not have an element of order $pq$ by Theorem~\ref{prop:Brauer_Line}. 

Now consider the alternating group $A_n$. Since $\ZZ A_n \subseteq \ZZ S_n$ is a subring, we only need to deal with those cases where $S_n$ possesses elements of order $pq$, but $A_n$ does not, i.e. $q = 2$ and $n = p + 2$ or $n = p + 3$. So in particular, $A_n$ has only one conjugacy class of elements of order $p$. The outer automorphism $\sigma$ of $A_n$ defines an admissible automorphism of the Brauer Tree of the principal $p$-block of $A_n$ by \cite[Lemma~3.1]{FeitBrauerTrees}. As $\sigma$ fixes the trivial character, it follows that it fixes every character in the principal block, cf. the definition of admissible automorphism in \cite[p.45]{FeitBrauerTrees}. A character $\chi$ is fixed by $\sigma$ if and only if it is rational. So the principal block contains only rational characters and its Brauer Tree is a line \cite[VII, Theorem~9.2]{FeitRep}.
Again Theorem~\ref{prop:Brauer_Line} can be used to see that there are no units of order $2p$ in $\V(\ZZ A_n)$. 
\end{proof}

Theorem~\ref{prop:Brauer_Line} can not only be used for symmetric and alternating groups, but is also strong enough to provide an affirmative answer to the Prime Graph Question for other almost simple groups; in this case for almost simple groups having sporadic simple group $Fi_{22}$ or $HN$ as socle.

\begin{figure}[ht!]
\caption{Prime graphs of the sporadic simple Harada–Norton group $HN$, the Fischer group $Fi_{22}$ and their automorphism groups.}\label{fig:PQ_HN_Fi22}
\begin{subfigure}{.45\textwidth}
\centering
\caption{$\Gamma(HN)= \Gamma(HN.2)$}
\label{fig:PQ_HN}
\begin{tikzpicture}
\node[label=north:{$2$}] at (0,4) (2){};
\node[label=north:{$3$}] at (-1.7320508, 3) (3){};
\node[label=north:{$5$}] at (1.7320508, 3) (5){};
\node[label=south:{$7$}] at (-1.7320508, 1) (7){};
\node[label=south:{$11$}] at (1.7320508, 1) (11){};
\node[label=south:{$19$}] at (0,0) (19){};
\foreach \p in {2,3,5,7,11,19}{
\draw[fill=black] (\p) circle (.075cm);
}
\draw (2)--(3);
\draw (2)--(5);
\draw (2)--(7);
\draw (2)--(11);
\draw (3)--(5);
\draw (3)--(7);
\draw (5)--(7);
\end{tikzpicture}
\end{subfigure}
\begin{subfigure}{.45\textwidth}
\centering
\caption{$\Gamma(Fi_{22})= \Gamma(Fi_{22}.2)$}
\label{fig:PQ_Fi22}
\begin{tikzpicture}
\node[label=north:{$2$}] at (0,4) (2){};
\node[label=north:{$3$}] at (-1.7320508, 3) (3){};
\node[label=north:{$5$}] at (1.7320508, 3) (5){};
\node[label=south:{$7$}] at (-1.7320508, 1) (7){};
\node[label=south:{$11$}] at (1.7320508, 1) (11){};
\node[label=south:{$13$}] at (0,0) (13){};
\foreach \p in {2,3,5,7,11,13}{
\draw[fill=black] (\p) circle (.075cm);
}
\draw (2)--(3);
\draw (2)--(5);
\draw (2)--(7);
\draw (2)--(11);
\draw (3)--(5);
\draw (3)--(7);
\end{tikzpicture}
\end{subfigure}
\end{figure}
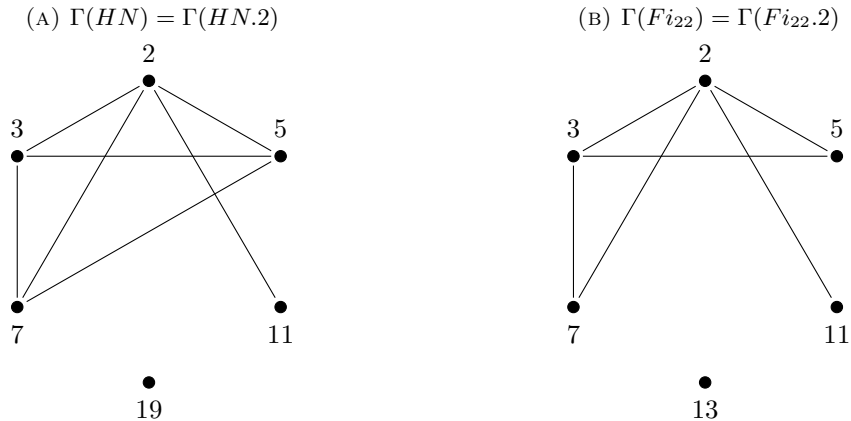

\begin{proof}[Proof of Theorem~\ref{th:Fi22HN}]
The outer automorphism groups of $Fi_{22}$ and $HN$ are both cyclic of order $2$. The prime graphs of both groups coincide with the prime graph of the corresponding automorphism group and are given in Figure~\ref{fig:PQ_HN_Fi22}. It can be read off from the ATLAS \cite{ATLAS} that $\text{Aut}(Fi_{22}) = Fi_{22}.2$ has cyclic Sylow $p$-subgroup of order $p$ such that there is only one class of $p$-elements in $\text{Aut}(Fi_{22})$ for $p \in \{7,11,13 \}$. Likewise, $\text{Aut}(HN) = HN.2$ has cyclic Sylow $p$-subgroup of order $p$ such that there is only one class of $p$-elements in $\text{Aut}(HN)$ for $p \in \{7,11,19\}$. For both groups and all primes mentioned the corresponding Brauer tree of the principal block is a line. This fact can be found in \cite{HissLux} or can be derived from the \textsf{GAP} CharacterTableLibarary \cite{CTblLib}. Hence by Theorem~\ref{prop:Brauer_Line} and the prime graphs of the groups in Figure~\ref{fig:PQ_HN_Fi22} it follows that the Prime Graph Question has a positive answer for any almost simple group with socle $Fi_{22}$ or $HN$.
\end{proof}

\begin{remark}
\begin{enumerate}[label=(\alph*)]
\item The Prime Graph Question for the sporadic simple group $Fi_{22}$ and its automorphism group was also settled by Verbeken in \cite{Brecht} using the HeLP-method and, to handle units of order 33, \cite[Theorem 1.2]{MargolisConway}, also a main ingredient in the proof of Theorem~\ref{prop:Brauer_Line}. Using similar arguments he also answered the question for $HN$ and its automorphism group. 
\item Concerning sporadic simple groups we note that Theorem~\ref{prop:Brauer_Line} suffices also to verify the Prime Graph Question for the following groups: The Mathieu group of degree 12, the second Janko group, the Higman-Sims group and the Suzuki group. This confirms results of Bovdi, Konovalov et al.
\end{enumerate}
\end{remark}

\noindent
\textbf{Acknowledgement:} We thank Jürgen Müller for answering our questions about Brauer Trees. We also thank the referee who helped to improve the readability of the paper.

\bibliographystyle{amsalpha}
\bibliography{brauer_line}

\end{document}